\documentclass[11pt]{article}

\usepackage[margin=1.15in]{geometry}
\usepackage{quoting}

\usepackage{stmaryrd}
\usepackage{amsfonts}
\usepackage{graphicx}
\usepackage{epstopdf}
\usepackage{amsmath}
\usepackage{cite}    
\usepackage{amsopn}
\usepackage{amsthm}
\usepackage{thmtools}
\usepackage{thm-restate}
\usepackage{bm}
\usepackage{amssymb}
\usepackage{algorithm}
\usepackage{algpseudocode}
\usepackage{subcaption}
\usepackage[percent]{overpic}

\usepackage[dvipsnames,table]{xcolor}
\definecolor{c0}{HTML}{ffb000}
\definecolor{c1}{HTML}{fe6100}
\definecolor{c2}{HTML}{dc267f}
\definecolor{c3}{HTML}{785ef0}
\definecolor{c4}{HTML}{648fff}
\definecolor{matlabblue}{RGB}{0,114,189}
\definecolor{matlaborange}{RGB}{217,83,25}
\definecolor{matlabyellow}{RGB}{237,177,32}
\definecolor{matlabpurple}{RGB}{126,47,142}
\definecolor{matlabgreen}{RGB}{119,172,48}
\definecolor{ibmblue}{RGB}{100,143,255}     
\definecolor{ibmorange}{RGB}{254,97,0}      
\definecolor{ibmgold}{RGB}{255,176,0}       
\definecolor{ibmmagenta}{RGB}{220,38,127}   
\usepackage{hyperref}
\hypersetup{
    colorlinks=true,      
    linkcolor=c2,       
    citecolor=c4,      
    filecolor=c4,    
    urlcolor=c4         
}
\usepackage[capitalize,nameinlink,noabbrev]{cleveref}
\crefformat{equation}{#2(#1)#3}
\crefmultiformat{equation}{#2(#1)#3}%
{ and~#2(#1)#3}{, #2(#1)#3}{ and~#2(#1)#3}

\usepackage{array}
\usepackage{colortbl}
\usepackage[labelfont=bf]{caption}
\usepackage[shortlabels]{enumitem}

\usepackage{tikz}
\usetikzlibrary{arrows.meta,positioning,fit,backgrounds,calc,patterns}
\usetikzlibrary{matrix}

\DeclareMathOperator{\tr}{tr}
\DeclareMathOperator{\Var}{Var}

\DeclareMathOperator{\range}{range}

\newcommand{\T}{\mathsf{T}}
\newcommand{\F}{\mathsf{F}}
\newcommand{\R}{\mathbb{R}}

\newcommand{\gaussian}{\operatorname{Gaussian}}
\newcommand{\Hsharp}{H_r^{\#}}
\newcommand{\Hsharpw}[1]{H_{r,#1}^{\#}}
\newcommand{\Hflat}{H_r^{\flat}}

\usepackage[labelfont=bf]{caption}

\usepackage[shortlabels]{enumitem}

\usepackage[dvipsnames]{xcolor}

\newtheorem{theorem}{Theorem}[section]
\newtheorem*{theorem*}{Theorem}

\newtheorem{fact}[theorem]{Fact}

\title{Hutch\#: Optimal non-adaptive Frobenius norm estimation}
\author{Tyler Chen\thanks{New York University Shanghai.
  \href{mailto:tyler.chen@nyu.edu}{tyler.chen@nyu.edu}}
\and Diana Halikias\thanks{New York University.
  \href{mailto:diana.halikias@nyu.edu}{diana.halikias@nyu.edu},
  \href{mailto:cmusco@nyu.edu}{cmusco@nyu.edu},
  \href{mailto:dup210@nyu.edu}{dup210@nyu.edu}}
\and Christopher Musco\footnotemark[2]
\and
David Persson\footnotemark[2]~\thanks{Flatiron Institute.
  \href{mailto:dpersson@flatironinstitute.org}{dpersson@flatironinstitute.org}}}

\begin{document}
\maketitle

\begin{abstract}

The Girard--Hutchinson estimator provides an extremely simple randomized estimate of the Frobenius norm of a matrix $\bm{A}$ that can only be accessed implicitly via matrix-vector products. In particular, if $\bm{\Omega}$ is a random Gaussian matrix with $r =  O(1/\varepsilon^2)$ columns, than $\frac{1}{r}\|\bm{A}\bm{\Omega}\|_\F^2$ provides a $(1\pm \varepsilon)$ multiplicative approximation to $\|\bm{A}\|_\F^2$  with high probability.

In this work, we introduce a closely related estimator, given by 
\begin{align*}
    {\frac{1}{r}\|\bm{A}\bm{\Omega}\|_\F^2 + \frac{1}{r}\|\bm{\Psi}^\T \bm{A}\|_\F^2 - \frac{1}{r^2}\|\bm{\Psi}^\T \bm{A}\bm{\Omega}\|_\F^2},
\end{align*}
where $\bm{\Psi}$ is a second, independent random Gaussian matrix with $r$ columns. We prove that this estimator yields a $(1\pm\varepsilon)$ multiplicative approximation to $\|\bm{A}\|_\F^2$ when $r = O(1/\varepsilon)$, a quadratic improvement over Girard--Hutchinson.
This dependence on $\varepsilon$ is optimal. Our method, which we call Hutch\# (pronounced ``Hutch sharp''), matches the complexity of the Hutch++ algorithm [Meyer, Musco, Musco, Woodruff, 2021].
However, unlike Hutch++, Hutch\# uses only \textit{non-adaptive} matrix-vector products with $\bm{A}$ and $\bm{A}^\T$ and   requires no orthogonalization or other advanced linear algebra steps. Thus, Hutch\# combines the simplicity of the Girard--Hutchinson estimator  and the optimal query complexity of Hutch++. 
\end{abstract}

\section{Introduction}
\label{sec:intro}
Consider a matrix $\bm{A} \in \mathbb{R}^{m \times n}$ implicitly accessible via matrix-vector products (matvecs):
\begin{equation*}
    \bm{x} \mapsto \bm{A}\bm{x}, \quad \bm{y} \mapsto \bm{A}^\T \bm{y}.
\end{equation*}
A fundamental task in linear algebra is to  compute an approximation, $F^2$, to the squared Frobenius norm $\|\bm{A}\|_\F^2 := \sum\limits_{i=1}^m \sum\limits_{j=1}^n \bm{A}_{i,j}^2$ satisfying 
\begin{equation}\label{eq:frob_estimation}
    \left|F^2 - \|\bm{A}\|_\F^2 \right| \leq \varepsilon \|\bm{A}\|_\F^2.
\end{equation}

This task arises when estimating the error of low-rank, hierarchical, and sparse matrix approximations~\cite{epperly2024efficient, gorman2019robust,smith2026adaptive, hausner2024neural, pritchard2025fast},  computing the Frobenius norm of implicit Jacobian and Hessian matrices in optimization pipelines~\cite{bai2021stabilizing, finlay2020train, tahmasebi2024universal},  learning structured linear operators~\cite{pmlr-v336-amsel26a}, as well as in many other applications \cite{haber2012effective}.

The simplest and most well-known method for Frobenius norm estimation is the Girard--Hutchinson estimator~\cite{girard1987trace,hutchinson}, which takes the form\footnote{In this work, we consider estimators that use independent Gaussians vectors. This is for simplicity of analysis. Similar bounds hold for estimators based on other distributions, such as vectors with Rademacher random entries, which were used in Hutchinson's original paper.}
\begin{align*}
H_r &:= \frac{1}{r}\|\bm{A}\bm{\Omega}\|_\F^2,  &
&\text{where $\bm{\Omega} \in \R^{n \times r}$ has i.i.d. standard Gaussian entries.}
\end{align*}
It is easy to show that this expression yields an unbiased estimator for $\|\bm{A}\|_\F^2$ with variance $\frac{2}{r}\| \bm{A} \|_{(4)}^4 \leq \frac{2}{r}\| \bm{A} \|_\F^4$, where $\|\cdot\|_{(4)}$ is the Schatten-4 norm \cite{avron2011randomized}.
It follows that, when $r = O(1/\varepsilon^2)$, the Girard--Hutchinson estimator has variance on the order of $\varepsilon^2 \|\bm{A}\|_\F^4$. Thus, by Chebyshev's inequality, it provides an approximation satisfying \cref{eq:frob_estimation} with high probability.

A strength of the Girard--Hutchinson estimator is its simplicity and  use of \textit{non-adaptive} matrix-vector products; we simply multiply by a set of random vectors that can all be chosen upfront. In contrast to adaptive algorithms, like Krylov subspace methods,  non-adaptive algorithms are beneficial in parallel and distributed settings, where matvecs can be executed concurrently, and in streaming settings, where $\bm{A}$ is too large to revisit or arrives through sequential updates~\cite{tropp2017practical, tropp2023randomized, clarkson2009numerical, dharangutte2021dynamic, al2026communication, madhavan2026flextrace}. Moreover, non-adaptivity is essential in some applications. Consider, for example, the task of estimating the Frobenius norm error between $\bm{A}$ and several different candidate approximations $\bm{B}_1, \bm{B}_2, \ldots, \bm{B}_k$. A non-adaptive method can \emph{reuse} matvecs with $\bm{A}$ to compute the necessary matvecs with $\bm{A} - \bm{B}_i$, while an adaptive method would need to use different matvecs for each, potentially incurring a $k\times$ penalty in error estimation cost. 

If one allows for adaptive methods, where the result of past matvecs can inform the construction of subsequent matvecs, the dependence on $\varepsilon$ can be improved. 
In particular, since $\|\bm{A}\|_\F^2 = \tr(\bm{A}^\T \bm{A})$, once can use trace estimators, where matvecs with $\bm{A}^\T \bm{A}$ are computed by sequential products with $\bm{A}$ and $\bm{A}^\T$.
Methods such as Hutch++, Nystr\"om++, and XTrace \cite{hutchpp, AHutchpp, epperly2024xtrace}  attain the same $(1 \pm \varepsilon)$ multiplicative guarantee of \cref{eq:frob_estimation} using just  $O(1/\varepsilon)$ adaptive matrix-vector products with $\bm{A}$ and $ \bm{A}^\T$.

It is known that the complexity of these methods is optimal for Frobenius norm estimation. No method can achieve the bound \cref{eq:frob_estimation} with fewer than $O(1/\varepsilon)$ matvecs, even if adaptivity is allowed \cite{meyer2024thesis}.\footnote{It is important that multiplication with both $\bm{A}$ and $\bm{A}^\T$ is allowed. Interestingly, if we only have matvec access to $\bm{A}$, the lower bound is $\Omega(1/\varepsilon^2)$ \cite{halikias2026transpose}, i.e., the Girard--Hutchinson estimator is optimal.} A natural question is if adaptivity is \emph{necessary} to achieve this optimal bound: 
\begin{quoting}[leftmargin=0.5em,rightmargin=0.5em,indentfirst=false]
\textbf{Q:} Is there a non-adaptive, optimal-complexity method for Frobenius norm estimation?
\end{quoting}

\subsection{An Optimal, Non-Adaptive Method}

We answer this question in the affirmative by proposing  a new  estimator for the Frobenius norm, which we call Hutch\#. This estimator takes the form:
\begin{equation}\label{eq:hutch_sharp}
    \Hsharp := {\frac{1}{r}\|\bm{A}\bm{\Omega}\|_\F^2 + \frac{1}{r}\|\bm{\Psi}^\T \bm{A}\|_\F^2 - \frac{1}{r^2}\|\bm{\Psi}^\T \bm{A}\bm{\Omega}\|_\F^2},
\end{equation}
where $\bm{\Omega} \in \R^{n \times r}$ and $\bm{\Psi} \in \R^{m \times r}$ are matrices with i.i.d. standard Gaussian entries. Hutch\# can be computed using $2r$ non-adaptive matrix-vector products with $\bm{A}$ and $\bm{A}^\T$. Moreover, it is nearly as simple as the original Girard--Hutchinson estimator.
Unlike Hutch++ and relatives, it does not require any advanced linear algebra kernels like orthogonalization or pseudoinverse. 

Our main theoretical result is that Hutch\# achieves the multiplicative guarantee of \cref{eq:frob_estimation} with an optimal $O(1/\varepsilon)$ queries. In particular, we show:

\begin{theorem}\label{thm:main}
    $\Hsharp(\bm{A})$ is an unbiased estimator of $\|\bm{A}\|_\F^2$, and its variance satisfies
    \begin{equation*}
        \Var\left(\Hsharp(\bm{A})\right) = \frac{2}{r^2}\|\bm{A}\|_{(4)}^4 +\frac{2}{r^2}\|\bm{A}\|_\F^4\leq \frac{4}{r^2}\|\bm{A}\|_\F^4.
    \end{equation*}
\end{theorem}
Setting $r = O(1/\varepsilon)$, we see that Hutch\# has variance on the order of $\varepsilon^2 \|\bm{A}\|_\F^4$, as desired. 
While the proof of \cref{thm:main} is not difficult, the Hutch\# estimator might seem unintuitive at first glance. To give some clarity, we observe that the method is actually quite closely related to the Hutch++ family of algorithms. These algorithms are based on estimating $\tr(\bm{A}^\T \bm{A})$ via a \emph{control-variate} method. In particular, a small number of matrix-vector products are used to obtain a low-rank approximation, $\bm{B}$, to $\bm{A}^\T \bm{A}$. Hutch++ then returns the estimate: 
\begin{align}
    \label{eq:hutchpp_template}
    \tr(\bm{B}) + \frac{1}{r}\tr\left(\bm{\Omega}^\T (\bm{A}^\T \bm{A} - \bm{B})\bm{\Omega}\right).
\end{align}
The variance improvement is obtained via a trade-off. Recall that the Girard--Hutchinson estimator has variance $\frac{2}{r}\|\bm{A}\|_{(4)}^4$. While this can always be upper bounded by $\frac{2}{r}\|\bm{A}\|_{F}^4$, it will be far smaller if $\bm{A}^\T \bm{A}$ has a flat spectrum. On the other hand, if $\bm{A}^\T \bm{A}$ has a quickly decaying spectrum, than $\bm{A}^\T \bm{A} - \bm{B}$ will be small, so the variance of the residual estimate, $\frac{1}{r}\tr\left(\bm{\Omega}^\T (\bm{A}^\T \bm{A} - \bm{B})\bm{\Omega}\right)$ is far smaller than the variance of a direct estimate for $\tr(\bm{A}^\T \bm{A})$.

Hutch++ and relatives compute the low-rank approximation, $\bm{B}$, using the randomized SVD~\cite{rsvd} or generalized Nystr\"om method~\cite{nakatsukasa2020fast}. For Frobenius norm estimation, these routines require adaptive matvecs with $\bm{A}$ and $\bm{A}^\T$. 
Hutch\# is obtained by replacing these methods with the much coarser low-rank approximation,  $\bm{B} = \frac{1}{r}\bm{A}^\T\bm{\Psi}\bm{\Psi}^\T \bm{A}$. The reader might recognize this as a standard ``randomized matrix multiplication'' approximation to $\bm{A}^\T\bm{A}$~\cite{DrineasKannanMahoney:2006,Sarlos:2006}. Plugging into \cref{eq:hutchpp_template}, we obtain
\begin{align*}
    \tr\left(\frac{1}{r}\bm{A}^\T\bm{\Psi}\bm{\Psi}^\T \bm{A}\right) + \frac{1}{r}\tr\left(\bm{\Omega}^\T \left(\bm{A}^\T \bm{A} - \frac{1}{r}\bm{A}^\T\bm{\Psi}\bm{\Psi}^\T \bm{A}\right)\bm{\Omega}\right) = \Hsharp(\bm{A}).
\end{align*}

Our main observation is that this  crude low-rank approximation suffices to achieve the same optimal rate as Hutch++ for the task of Frobenius norm estimation. The proof is straightforward, and can be found in \Cref{sec:analysis}.

\subsection{Additional Results}
The implementation of Hutch\# is particularly simple; other than the two Gaussian sketches $\bm A\bm\Omega$ and $\bm A^\T\bm\Psi$, it requires only cheap Frobenius norms and the small cross-product $\bm\Psi^\T\bm A\bm\Omega$, with no orthogonalization, pseudoinverse, or adaptive linear-algebraic step. We also study two variants of Hutch\# that trade some additional computation for other advantages. First, we introduce the weighted family of estimators parameterized by $0 \leq \alpha \leq 1$:
\begin{equation}\label{eq:weighted}
    \Hsharpw{\alpha}(\bm{A}) :=
    \frac{1+\alpha}{2r}\|\bm{A}\bm{\Omega}\|_\F^2 + \frac{1+\alpha}{2r}\|\bm{\Psi}^\T \bm{A}\|_\F^2 - \frac{\alpha}{r^2}\|\bm{\Psi}^\T \bm{A}\bm{\Omega}\|_\F^2.
\end{equation}
$\Hsharpw{\alpha}$ is a convex combination of Hutch\# and a symmetrized Girard--Hutchinson estimator, equal to the former when $\alpha = 1$ and the latter when $\alpha = 0$. It can be shown that the choice of $\alpha$ that minimizes $\Var(\Hsharpw{\alpha}(\bm{A}))$ depends only on the effective-rank quantity $R:=\|\bm A\|_\F^4/\|\bm A\|_{(4)}^4$, where $\|\bm{A}\|_{(4)}$ denotes the Schatten-4 norm. This motivates a ``spectrum-aware'' estimator that estimates $R$ from the same sketches used in~\cref{eq:weighted}. In experiments, we observe that this improved estimator universally outperforms  vanilla Hutch\# and the Girard--Hutchinson estimator in all cases. See \Cref{sec:experiments} for numerical results.

Second, we introduce Hutch$\flat$ (pronounced ``Hutch flat''), a non-adaptive estimator based on the generalized Nystr\"om approximation $\widetilde{\bm A}_{\mathrm{GN}}
:=
\bm A\bm\Omega
(\bm\Psi^\T\bm A\bm\Omega)^\dagger
\bm\Psi^\T\bm A$ \cite{nakatsukasa2020fast,clarkson2009numerical,woodruff2014sketching}. Defining $\widehat{\bm B}_\flat
:=
\widetilde{\bm A}_{\mathrm{GN}}^\T
\widetilde{\bm A}_{\mathrm{GN}}$, we estimate $\|\bm A\|_\F^2$ by
\begin{equation}\label{eq:hutch_flat}
    \Hflat :=
\tr(\widehat{\bm B}_\flat)
+
\frac{1}{r}\tr\left(
\bm G^\T(\bm A^\T\bm A-\widehat{\bm B}_\flat)\bm G
\right),
\end{equation}
where $\bm{G} \in \R^{n \times r}$ is a third random Gaussian matrix. All necessary matvecs can be performed non-adaptively. Hutch$\flat$ requires more postprocessing than Hutch\#, including a pseudoinverse, but can achieve much lower variance when $\bm{A}$'s singular values  decay rapidly.

\subsection{Notation}
Before presenting an analysis of Hutch\#, we describe notation used throughout the paper.
For $\bm{A} \in \R^{m \times n}$, we write $\bm{A}_{i,j}$ for its $(i,j)$ entry, $\bm{A}^\T$ for its transpose, and $\sigma_1(\bm{A}) \geq \sigma_2(\bm{A}) \geq \cdots \geq \sigma_{\min\{m,n\}}(\bm{A}) \geq 0$ for its singular values.
We write $\tr(\bm{A})$ for the trace of a square matrix.
We let $\|\bm{A}\|_\F$ denote the Frobenius norm, $\|\bm{A}\|_\F := \big(\sum_{i,j} \bm{A}_{i,j}^2\big)^{1/2} = \big(\sum_{j} \sigma_j(\bm{A})^2\big)^{1/2} = \tr(\bm{A}^\T \bm{A})^{1/2}$, and $\|\bm{A}\|_{(4)}$ the Schatten-4 norm,
\begin{equation}\label{eq:schatten4}
    \|\bm{A}\|_{(4)} := \bigg(\sum_{j} \sigma_j(\bm{A})^4\bigg)^{1/4},
    \quad\text{so}\quad
    \|\bm{A}\|_{(4)}^4 = \|\bm{A}^\T \bm{A}\|_\F^2 = \|\bm{A}\bm{A}^\T\|_\F^2.
\end{equation}
We always have $\|\bm{A}\|_{(4)} \leq \|\bm{A}\|_\F$, with equality when $\bm{A}$ has rank one.
At the other extreme, when $\bm{A}$ has $k$ equal non-zero singular values, $\|\bm{A}\|_{(4)}^4 = \frac{1}{k}\|\bm{A}\|_\F^4$.
So, the ratio $\|\bm{A}\|_\F^4 / \|\bm{A}\|_{(4)}^4 \in [1, \min\{m,n\}]$ can be viewed as a measure of the ``flatness'' of $\bm{A}$'s spectrum.

We write $\bm{\Omega} \sim \gaussian(n,r)$ to indicate that $\bm{\Omega} \in \R^{n \times r}$ is a random matrix with i.i.d. standard normal entries.
$\mathbb{E}[X]$ and $\Var(X)$ denote the expectation and variance of a random variable $X$, and $\mathbb{E}[X \mid Y]$ and $\Var(X \mid Y)$ the same conditioned on a second random variable $Y$.
Our variance bounds use the law of total variance, $\Var(X) = \mathbb{E}\left[\Var(X \mid Y)\right] + \Var\left(\mathbb{E}[X \mid Y]\right)$.

\section{Main Analysis}
\label{sec:analysis}
In this section, we analyze the Hutch\# estimator, $\Hsharp$, described in \cref{eq:hutch_sharp}. As discussed, our analysis is based on the observation that: 
\begin{equation}\label{eq:hutch++relation}
     \Hsharp(\bm{A}) = \tr\left(\frac{1}{r}\bm{A}^\T \bm{\Psi} \bm{\Psi}^\T \bm{A} \right) + \frac{1}{r}\tr\left( \bm{\Omega}^\T \left( \bm{A}^\T \bm{A}  - \frac{1}{r} \bm{A}^\T \bm{\Psi} \bm{\Psi}^\T \bm{A} \right)\bm{\Omega}\right).
\end{equation}
The analysis of Hutch\# then follows immediately from the following well-known facts:
\begin{fact}[See, e.g., \cite{avron2011randomized}]\label{fact:hutch}
Let $\bm{B} \in \mathbb{R}^{n \times n}$ be symmetric and let $\bm{\Omega} \sim \gaussian(n,r)$. 
Then,
\begin{equation}\label{eq:standardequalitues}
    \mathbb{E}\left[\frac{1}{r} \tr\left(\bm{\Omega}^\T \bm{B} \bm{\Omega}\right)\right] = \tr(\bm{B}), \qquad \Var\left(\frac{1}{r}\tr\left(\bm{\Omega}^\T \bm{B} \bm{\Omega}\right)\right) = \frac{2}{r}\|\bm{B}\|_\F^2.
\end{equation}
\end{fact}

\begin{fact}[See, e.g. \cite{puchkin2025}]\label{fact:frobeniusbound}
Let $\bm{A} \in \mathbb{R}^{m \times n}$ and $\bm{\Psi} \sim \gaussian(m,r)$. Then,
\begin{equation*}
     \mathbb{E}\left\|\bm{A}^\T \bm{A} - \frac{1}{r}\bm{A}^\T \bm{\Psi} \bm{\Psi}^\T \bm{A}\right\|_\F^2 = \frac{1}{r}\|\bm{A}\|_{(4)}^4 +\frac{1}{r}\|\bm{A}\|_\F^4.
\end{equation*}
\end{fact}
We remark that our targeted $O(1/\varepsilon)$ rate for Hutch\# can be obtained by replacing \Cref{fact:frobeniusbound} with the less precise upper bound, $\mathbb{E}\left\|\bm{A}^\T \bm{A} - \frac{1}{r}\bm{A}^\T \bm{\Psi} \bm{\Psi}^\T \bm{A}\right\|_\F^2 \leq \frac{2}{r}\|\bm{A}\|_\F^4$. This is the standard variance bound that arises in the analysis of ``approximate matrix multipication'' methods in Randomized Numerical Linear Algebra (RandNLA) \cite{DrineasKannanMahoney:2006,Sarlos:2006}.

\begin{proof}[Proof of \cref{thm:main}]
    From \cref{eq:hutch++relation} and \cref{fact:hutch}, it is immediate that $\Hsharp(\bm{A})$ is unbiased. 
    We proceed with computing the variance. 
    By expressing $\Hsharp(\bm{A})$ as in \cref{eq:hutch++relation} and using the law of total variance, we have
    \begin{align*}
        \Var(\Hsharp(\bm{A})) &= \mathbb{E}\left[\Var\left(\Hsharp(\bm{A}) \middle| \bm{\Psi}\right)\right] + \Var\left(\mathbb{E}\left[\Hsharp(\bm{A}) \middle| \bm{\Psi}\right]\right)\\
        &=\frac{2}{r}\mathbb{E}\left\|\bm{A}^\T \bm{A} - \frac{1}{r}\bm{A}^\T \bm{\Psi} \bm{\Psi}^\T \bm{A}\right\|_\F^2 + \Var(\|\bm{A}\|_\F^2) \tag{\Cref{fact:hutch}}\\
        &=\frac{2}{r^2}\|\bm{A}\|_{(4)}^4 +\frac{2}{r^2}\|\bm{A}\|_\F^4. \tag{\Cref{fact:frobeniusbound} and $\Var(\|\bm{A}\|_\F^2) = 0$}
    \end{align*}
    The final inequality in the statement of \Cref{thm:main} follows from $\|\bm{A}\|_{(4)}\leq \|\bm{A}\|_\F$.
\end{proof}

We note that, with \Cref{thm:main} in place, we can obtain probability bounds on the error of Hutch\# via Chebyshev's inequality. In particular, we have that:
\begin{align*}
    \Pr\left[\left|\Hsharp(\bm{A}) - \|\bm{A}\|_\F^2\right| \geq \varepsilon \|\bm{A}\|_\F^2\right]
    \leq \frac{\Var\left(\Hsharp(\bm{A})\right)}{\varepsilon^2 \|\bm{A}\|_\F^4}
    \leq \frac{4}{\varepsilon^2 r^2}.
\end{align*}
So, for any $\delta \in (0,1)$, taking $r \geq {2}/({\varepsilon\sqrt{\delta})}$ suffices for Hutch\# to satisfy \cref{eq:frob_estimation} with probability at least $1-\delta$.
That is, $O\left(\frac{1}{\varepsilon\sqrt{\delta}}\right)$ matrix-vector products suffice to achieve the guarantee of \cref{eq:frob_estimation}.

We note that the  $1/\sqrt{\delta}$ dependence on the failure probability can be improved to $\log(1/\delta)$ one using the standard median-of-means trick \cite{alon1999space}. We can issue $O(\log(1/\delta))$ independent runs of Hutch\#, each with $r = O(1/\varepsilon)$, and return the median of the resulting estimates. By a standard Chernoff bound, we will obtain error $\epsilon \|\bm{A}\|_{\F}^2$ with probability at least $1-\delta$.

\section{An Improved Weighted Estimator}
\label{sec:weighted}

Hutch\# requires $O(1/\varepsilon)$ matvecs in the worst case to achieve  $\epsilon \|\bm{A}\|_{\F}^2$ error, compared to  $O(1/\varepsilon^2)$ for Girard--Hutchinson.
However, Hutch\# does not always outperform Girard--Hutchinson.
Indeed, recall that the variance of Girard--Hutchinson is on the order of $\frac{1}{r}\|\bm{A}\|_{(4)}^4$ while the variance of Hutch\# is on the order of $\frac{1}{r^2}\|\bm{A}\|_\F^4$. When $\bm{A}$ has a flat spectrum $\|\bm{A}\|_{(4)}^4 \approx \frac{1}{n}\|\bm{A}\|_\F^4$, so we can easily have that $\frac{1}{r}\|\bm{A}\|_{(4)}^4 < \frac{1}{r^2}\|\bm{A}\|_\F^4$ for most values of $r$. 

To address this issue, we introduce an alternative estimator that interpolates between Hutch\# and Girard--Hutchinson, always providing better variance than either. 
In particular, let $H_r(\bm{A}) := \frac{1}{2r}\left(\|\bm{A}\bm{\Omega}\|_\F^2 + \|\bm{\Psi}^\T\bm{A}\|_\F^2\right)$ be a symmetrized Girard--Hutchinson estimator. For $\alpha \in [0,1]$, we consider the convex combination of $H_r$ and $\Hsharp$ given by:
\begin{align*}
    \Hsharpw{\alpha}(\bm{A})
    &:= (1-\alpha)\, H_r(\bm{A}) + \alpha\, \Hsharp(\bm{A})\\
    &\phantom{:}= \frac{1+\alpha}{2r}\|\bm{A}\bm{\Omega}\|_\F^2 + \frac{1+\alpha}{2r}\|\bm{\Psi}^\T \bm{A}\|_\F^2 - \frac{\alpha}{r^2}\|\bm{\Psi}^\T \bm{A}\bm{\Omega}\|_\F^2,
\end{align*}
A direct computation reveals the variance of this estimator as:
\begin{equation*}
    \Var(\Hsharpw{\alpha}(\bm{A})) = \frac{(1-\alpha)^2}{r} \|\bm{A}\|_{(4)}^4 + \frac{2\alpha^2}{r^2} \left(\|\bm{A}\|_{(4)}^4 + \|\bm{A}\|_\F^4 \right). 
\end{equation*}

Let $R = \|\bm{A}\|_\F^4 / \|\bm{A}\|_{(4)}^4$ and $T = \frac{R+1}{r}$. 
The variance of $\Hsharpw{\alpha}$ is minimized for $\alpha^* = \frac{1}{1+2T}\in [0,1]$.
Since both Girard--Hutchinson and Hutch\# are special cases of this interpolating estimator, $\Hsharpw{\alpha^*}$ has smaller variance than each of these estimators.
\subsection{Practical Implementation}\label{sec:heuristic}
While we cannot efficiently compute $\alpha^*$, the quantity can be easily approximated using non-adaptive matvec queries. In particular, we need to compute approximations to $\|\bm{A}\|_\F^4$ and $\|\bm{A}\|_{(4)}^4$ so that we can approximate the effective-rank $R$. The former can be approximated using the squared Girard--Hutchinson estimator:
\begin{align}
    \label{eq:frob_est}
    \frac{1}{r^2}\|\bm{A}\bm{\Omega}\|_\F^4 \approx \|\bm{A}\|_\F^4.
\end{align}
To approximate the Schatten-4 norm, we use that $\|\bm{A}\|_{(4)}^4 = \|\bm{A}^\T\bm{A}\|_\F^2$. Ideally, we would like to apply the standard Girard--Hutchinson estimator, $\frac{1}{r}\|\bm{A}^\T\bm{A}\bm{\Omega}\|_\F^2$, but this would require adaptive matrix-vector products. 
Instead, we split the sketch in half: let $\bm{\Omega}_1$ contain the first $r/2$ columns of $\bm{\Omega}$ and let $\bm{\Omega}_2$ contain the remaining $r/2$, so that $\bm{\Omega}_1$ and $\bm{\Omega}_2$ are independent Gaussian matrices. 
Then $\frac{2}{r}\|\bm{\Omega}_1^\T\bm{A}^\T\bm{A}\bm{\Omega}_2\|_\F^2$ is an unbiased estimator for $\frac{2}{r}\|\bm{A}^\T\bm{A}\bm{\Omega}_2\|_\F^2$, which in turn is an unbiased estimator for $\|\bm{A}^\T\bm{A}\|_\F^4$. So we obtain the non-adaptive estimator:
\begin{align}
    \label{eq:schatten_est}
\|\bm{A}\|_{(4)}^4 \approx \frac{4}{r^2}\|\bm{\Omega}_1^\T\bm{A}^\T\bm{A}\bm{\Omega}_2\|_\F^2.
\end{align}
For both \cref{eq:frob_est} and \cref{eq:schatten_est}, it is not hard to check that setting $r = O(1)$ yields a constant-factor multiplicative approximation to $\|\bm{A}\|_\F^4$ and $\|\bm{A}\|_{(4)}^4$ with constant probability. Such approximations suffice to obtain $\widehat\alpha$ for which $\Var(\Hsharpw{\widehat\alpha}(\bm{A}))$ is within a constant factor of $\Var(\Hsharpw{\alpha^*}(\bm{A}))$. 

Concretely, suppose that we obtain estimates $\widehat{S}$ and $\widehat{F}$ that satisfy, for $c \geq 1$,
\begin{align*}
c^{-1} \| \bm{A} \|_{(4)}^4 &\leq \widehat{S} \leq c \|\bm{A} \|_{(4)}^4
& &\text{and}&
c^{-1} \| \bm{A} \|_{\F}^4 &\leq \widehat{F} \leq c \|\bm{A} \|_{\F}^4.
\end{align*}
Define $\widehat{T} := \frac{\widehat{F}/\widehat{S}+1}{r}$ and $\widehat{\alpha} = \frac{1}{1+2\widehat{T}}$. 
Then we have:
\begin{align}
    \label{eq:noisy_estimate_variance_bound}
\Var(\Hsharpw{\widehat\alpha}(\bm{A})) \leq \frac{(c+c^{-1})^2}{4}\Var(\Hsharpw{\alpha^*}(\bm{A})).
\end{align}
To see this, observe that the variance of $\Hsharpw{\alpha}$ can be written compactly as
\begin{align*}
    \Var(\Hsharpw{\alpha}(\bm{A})) = \frac{\|\bm{A}\|_{(4)}^4}{r}\, h(\alpha),
    \qquad\text{where}\qquad
    h(\alpha) := (1-\alpha)^2 + 2T\alpha^2.
\end{align*}
$\alpha^* = 1/(1+2T)$ minimizes $h$, and $h(\alpha^*) = \frac{2T}{2T+1}$.
Our assumptions on $\widehat{S}$ and $\widehat{F}$ give
$c^{-2}(R+1) \leq c^{-2}R + 1 \leq \widehat{F}/\widehat{S} + 1 \leq c^{2}R + 1 \leq c^{2}(R+1)$,
so $\widehat{T} = \kappa T$ for some $\kappa \in [c^{-2},c^{2}]$.
It follows that, 
\begin{align*}
    \frac{h(\widehat\alpha)}{h(\alpha^*)}
    = \frac{(2\kappa^2 T + 1)(2T+1)}{(1+2\kappa T)^2}
    = 1 + \frac{2(\kappa-1)^2 T}{(1+2\kappa T)^2}
    \leq 1 + \frac{(\kappa-1)^2}{4\kappa}
    = \frac{1}{4}\left(\sqrt{\kappa} + \frac{1}{\sqrt{\kappa}}\right)^2,
\end{align*}
where the inequality follows from $(1+2\kappa T)^2 \geq 8\kappa T$.
The final expression is maximized at $\kappa = c^{\pm2}$, which gives the stated bound in \cref{eq:noisy_estimate_variance_bound}.

We remark that, in our experiments, we do not use fresh matrix-vector products to estimate $\|\bm{A}\|_\F^4$ and $\|\bm{A}\|_{(4)}^4$ when calculating $\hat{\alpha}$. Instead, we simply reuse the same Gaussian sketch, $\bm{A}\bm{\Omega}$, that we used to compute $\Hsharpw{\alpha}(\bm{A})$. \textcolor{black}{Additional effort would be necessary to ensure that this does not lead to any dependency issues in the theoretical analysis, but the matvec reuse seems to show no issues experimentally.}
\section{An Estimator Based on the Generalized Nystr\"om Method}
\label{sec:gennystrompp_variance}

In this section, we introduce a final non-adaptive method for Frobenius norm approximation based on the generalized Nystr\"om method, which is a randomized low-rank approximation algorithm that only requires non-adaptive matrix-vector products with $\bm{A}$ and $\bm{A}^\T$ \cite{nakatsukasa2020fast,clarkson2009numerical,woodruff2014sketching}. This alternative method, which we call Hutch$\flat$ (``Hutch flat''), sacrifices some of the simplicity of Hutch\#, although it is by no means complicated. However, Hutch$\flat$ has the advantage that it can perform better than Hutch$\#$ for matrices that have rapidly decaying spectra. 

Concretely, Hutch$\flat$  requires three random Gaussian sketching matrices to estimate the Frobenius norm of a matrix $\bm A\in\mathbb R^{m\times n}$. In particular, for any integer $r\geq 16$, we draw\footnote{The sketch sizes $(2r,4r,2r)$ are chosen to keep the analysis simple. We have not attempted to optimize these proportions. In any case, all sketches should be chosen with $O(r)$ columns.}:
\begin{align*}
    \bm\Omega&\sim\gaussian(n,2r), &
    \bm\Psi&\sim\gaussian(m,4r), &
    \bm G&\sim\gaussian(n,2r).
\end{align*} 
Define the generalized Nystr\"om approximation
\begin{equation}\label{eq:gn_approximation}
    \bm B=\bm A\bm\Omega
      (\bm\Psi^\T\bm A\bm\Omega)^\dagger\bm\Psi^\T\bm A.
\end{equation}
We then define the Hutch$\flat$ estimator as: 
\begin{equation}\label{eq:gnpp_estimator}
    \Hflat(\bm A)
    =\|\bm B\|_\F^2
      +\frac{1}{2r}\left(\|\bm A\bm G\|_\F^2-\|\bm B\bm G\|_\F^2\right).
\end{equation}
This estimator requires $4r$ non-adaptive products with $\bm A$ and $4r$ non-adaptive products with $\bm A^\T$. Importantly, $\bm{B}\bm{G}$ is computable directly from $\bm{A}\bm{\Omega}$, $\bm{\Psi}^\T\bm{A}$, and $\bm{AG}$.

Our main theoretical result of this section is that $\Hflat$ satisfies the following variance bound.
\begin{theorem}\label{thm:gnpp_variance}
Let $r\geq 16$ and let $[\bm A]_r$ denote the best rank-$r$ approximation to $\bm A$.
The estimator $\Hflat(\bm A)$ is unbiased and its variance satisfies
\begin{align*}
    \Var(\Hflat(\bm A))
    \leq\frac{90}{r^2}\|\bm A\|_\F^2
         \|\bm A-[\bm A]_r\|_\F^2
    \leq\frac{90}{r^2}\|\bm A\|_\F^4.
\end{align*}
\end{theorem}
This bound matches \Cref{thm:main} up to constants (and we did not attempt to optimize the constant here). In particular, setting $r = O(1/\varepsilon)$, we obtain variance $\leq \varepsilon^2 \|\bm{A}\|_\F^4$, so can obtain a relative error approximation to the Frobenius norm with high probability. However, the bound can be much tighter than \Cref{thm:main} 
when $\bm{A}$ has spectral decay, and thus $\|\bm A-[\bm A]_r\|_\F^2 \ll \|\bm{A}\|_\F^2$. Indeed, as will be shown in \Cref{sec:experiments}, $\Hflat(\bm A)$ often outperforms Hutch\# for this reason.

\subsection{Analysis}
\label{sec:gnpp_analysis}
As in our analysis of Hutch$\sharp$, to prove \Cref{thm:gnpp_variance}, we require a bound on the approximation error $\|\bm A^\T\bm A-\bm B^\T\bm B\|_\F^2$, where $\bm{B}$ is the generalized Nystr\"om low-rank approximation from \cref{eq:gn_approximation}. We prove the following bound in \Cref{app:gn_gram} using relatively standard tools from the randomized numerical linear algebra literature:
\begin{restatable}{lemma}{GNGramLemma}\label{lemma:gn_gram}
Let $\bm B$ be as in \cref{eq:gn_approximation}.
For $r\geq16$,
\begin{equation}\label{eq:gnpp_gram_bound}
    \mathbb E\|\bm A^\T\bm A-\bm B^\T\bm B\|_\F^2
    \leq 35\left(\|\bm A-[\bm A]_r\|_{(4)}^2
    +\frac{\|\bm A-[\bm A]_r\|_\F^2}{\sqrt r}\right)^2 +\frac{20}{r}\|\bm A\|_\F^2\|\bm A-[\bm A]_r\|_\F^2.
\end{equation}
\end{restatable}
This lemma suffices to prove the main result.
\begin{proof}[Proof of \Cref{thm:gnpp_variance}]
By \Cref{fact:hutch} and the independence of $\bm G$ and $\bm B$ we immediately have that $
\mathbb E[\Hflat(\bm A)\mid\bm B]=\|\bm A\|_\F^2$. I.e., the Hutch$\flat$  estimator is unbiased as claimed.

The law of total variance and \Cref{lemma:gn_gram} then give
\begin{align}
    \label{eq:gn_almost_there}
    \Var(\Hflat(\bm A))
    &=\frac1r\,\mathbb E\|\bm A^\T\bm A-\bm B^\T\bm B\|_\F^2\\
    &\leq\frac{35}{r}\left(\|\bm A-[\bm A]_r\|_{(4)}^2
    +\frac{\|\bm A-[\bm A]_r\|_\F^2}{\sqrt r}\right)^2 +\frac{20}{r^2}\|\bm A\|_\F^2\|\bm A-[\bm A]_r\|_\F^2.
\end{align}
Since
$r\sigma_{r+1}(\bm A)^2\leq\|[\bm A]_r\|_\F^2$, we have
\begin{align*}
\|\bm A-[\bm A]_r\|_{(4)}^4
&\leq\frac{\|[\bm A]_r\|_\F^2\|\bm A-[\bm A]_r\|_\F^2}{r}.
\end{align*}
Combined with $(a+b)^2\leq2(a^2+b^2)$ and orthogonality of the
truncated SVD, we obtain:
\begin{align*}
\left(\|\bm A-[\bm A]_r\|_{(4)}^2+\frac{\|\bm A-[\bm A]_r\|_\F^2}{\sqrt r}\right)^2
&\leq\frac{2}{r}\|\bm A-[\bm A]_r\|_\F^2
\left(\|[\bm A]_r\|_\F^2+\|\bm A-[\bm A]_r\|_\F^2\right)\\
&=\frac{2}{r}\|\bm A\|_\F^2\|\bm A-[\bm A]_r\|_\F^2.
\end{align*}
Plugging into \cref{eq:gn_almost_there}, we obtain:
\begin{align*}
\Var(\Hflat(\bm A)) \leq (70+20)\frac{\|\bm A\|_\F^2\|\bm A-[\bm A]_r\|_\F^2}{r^2}
&=\frac{90}{r^2}\|\bm A\|_\F^2\|\bm A-[\bm A]_r\|_\F^2.\qedhere
\end{align*}
\end{proof}

\section{Numerical experiments}
\label{sec:experiments}
In this section, we provide an initial numerical comparison of the performance of our new Hutch\# estimator and its variants to existing implicit trace estimators.
\subsection{Hutch\# and Weighted Hutch\#}
We start by comparing our vanilla Hutch\# and its weighted variant from \Cref{sec:weighted} to the standard Girard--Hutchinson estimator, which, to the best of our knowledge, is the only existing implicit Frobenius norm estimator that only uses non-adaptive matrix-vector products. 

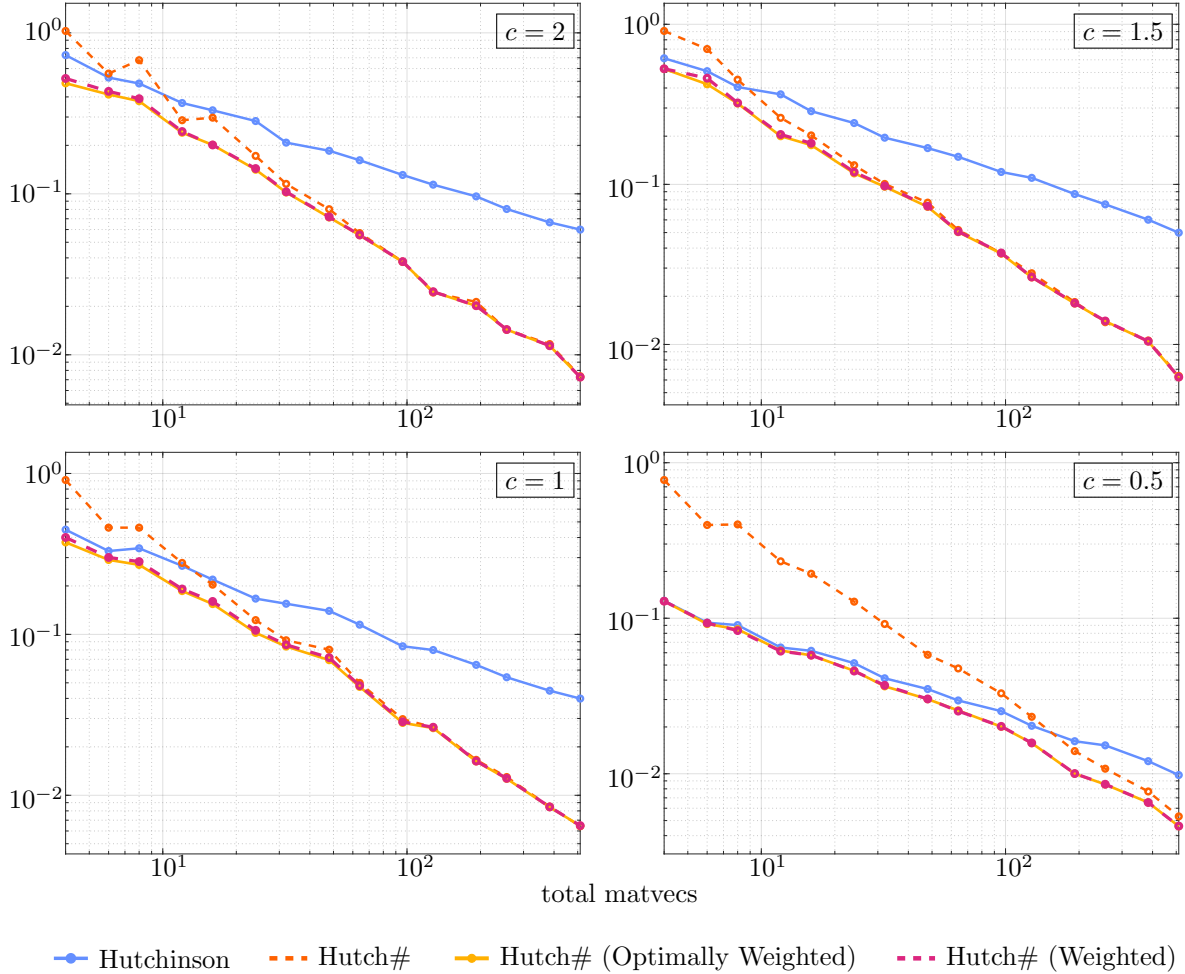
\begin{figure}[b!]
\centering

\setlength{\tabcolsep}{14pt}
\renewcommand{\arraystretch}{1.0}

\begin{tabular}{cc}

\begin{overpic}[width=0.44\textwidth]
    {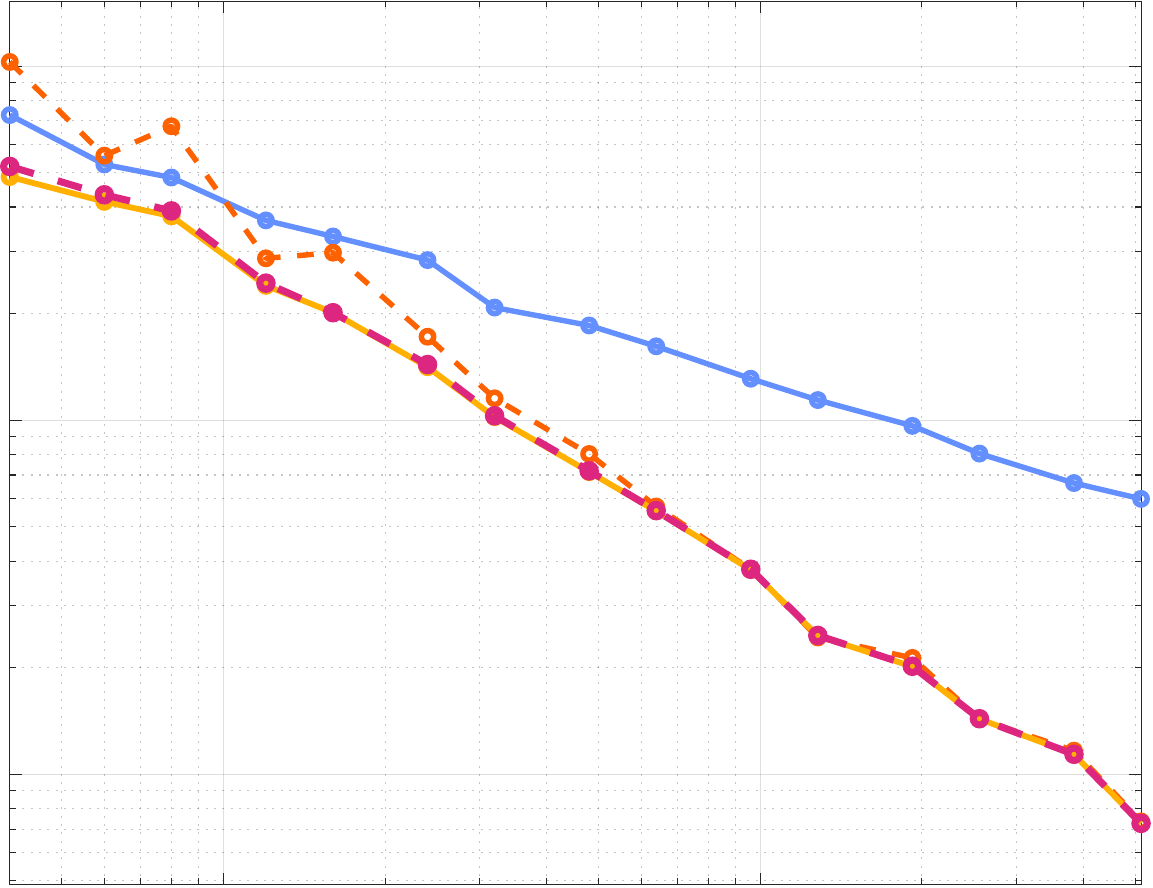}

    \put(83,70){\fbox{\small $c=2$}}

    \put(-7,70){\small $10^{0}$}
    \put(-10,39){\small $10^{-1}$}
    \put(-10,8){\small $10^{-2}$}

    \put(17,-4){\small $10^{1}$}
    \put(64,-4){\small $10^{2}$}

\end{overpic}
&
\begin{overpic}[width=0.44\textwidth]
    {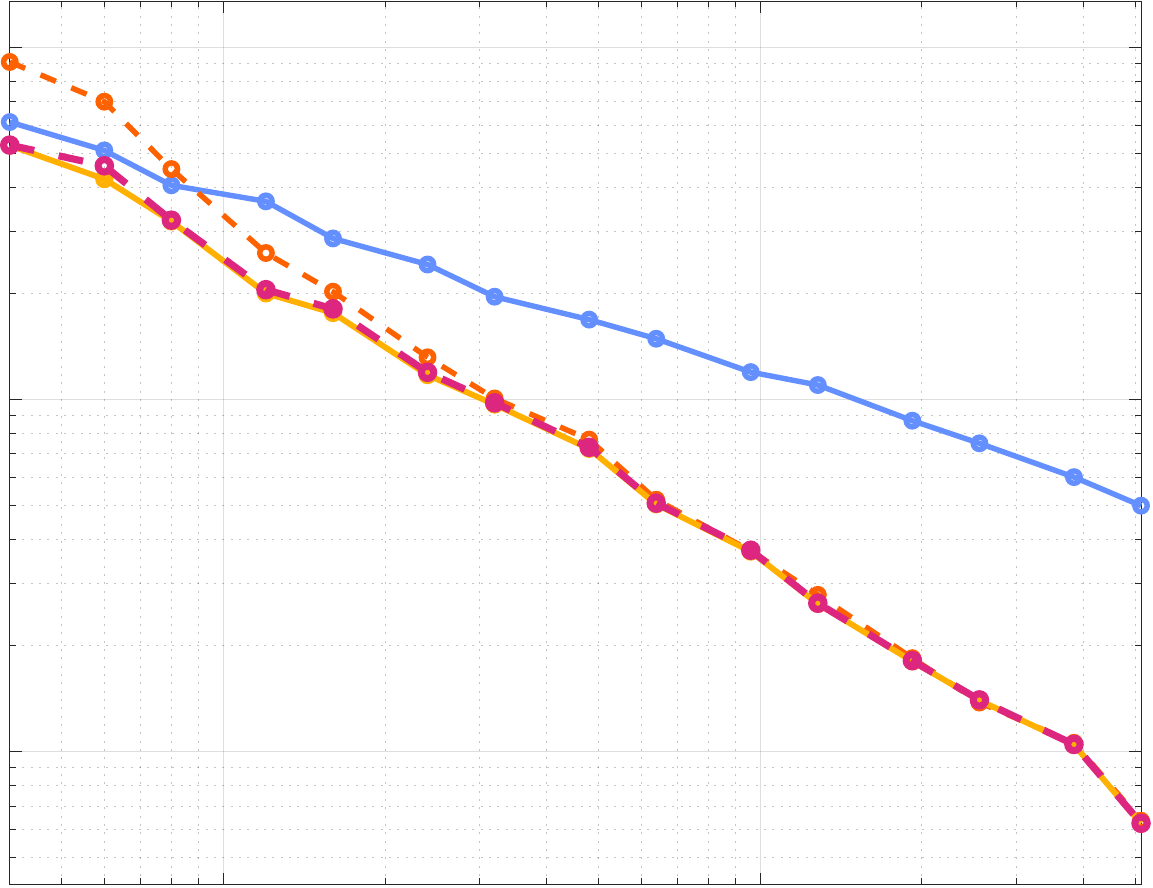}

    \put(79,70){\fbox{\small $c=1.5$}}

    \put(-7,71){\small $10^{0}$}
    \put(-10,40){\small $10^{-1}$}
    \put(-10,9){\small $10^{-2}$}

    \put(17,-4){\small $10^{1}$}
    \put(64,-4){\small $10^{2}$}

\end{overpic}

\\[1.2em]

\begin{overpic}[width=0.44\textwidth]
    {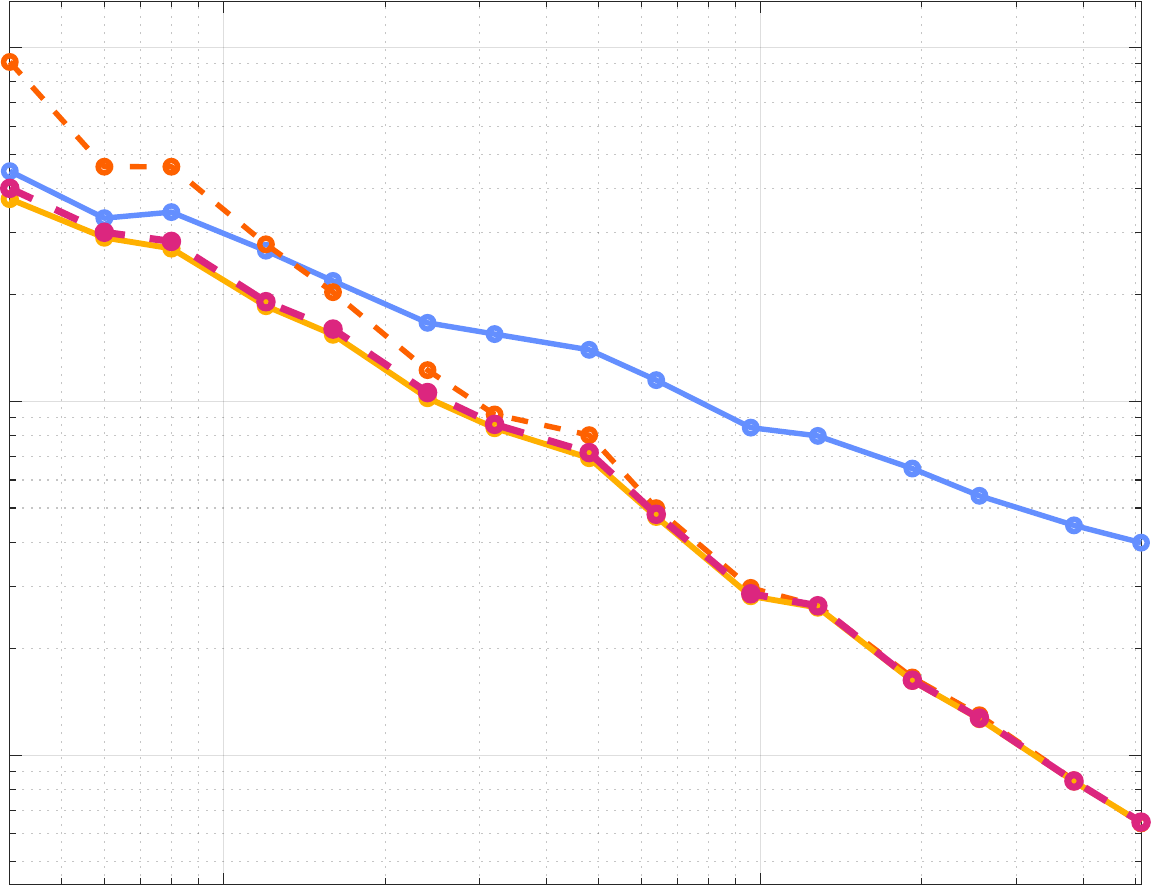}

    \put(83,70){\fbox{\small $c=1$}}

    \put(-7,71){\small $10^{0}$}
    \put(-10,40){\small $10^{-1}$}
    \put(-10,9){\small $10^{-2}$}

    \put(17,-4){\small $10^{1}$}
    \put(64,-4){\small $10^{2}$}

\end{overpic}
&
\begin{overpic}[width=0.44\textwidth]
    {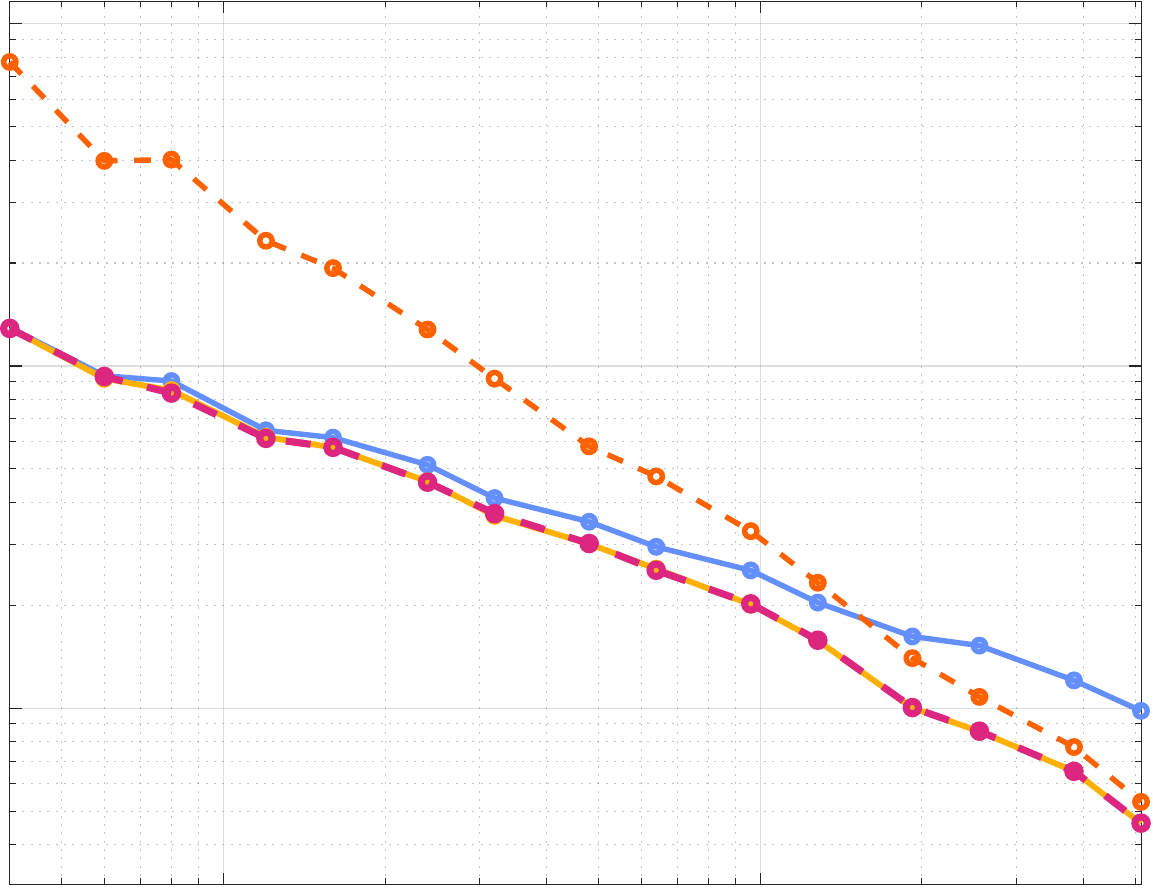}

    \put(79,70){\fbox{\small $c=0.5$}}

    \put(-7,73){\small $10^{0}$}
    \put(-10,43){\small $10^{-1}$}
    \put(-10,13){\small $10^{-2}$}

    \put(17,-4){\small $10^{1}$}
    \put(64,-4){\small $10^{2}$}

\end{overpic}

\end{tabular}

\vspace{0.5em}

{\small total matvecs}


\vspace{0.7em}

\begin{tikzpicture}[baseline]

\draw[
    ibmblue,
    line width=1.3pt
]
(0.25,0) -- (0.7,0);
\draw[
    ibmblue,
    line width=1pt,
    fill=ibmblue
]
(0.475,0) circle (1.5pt);
\node[anchor=west] at (0.71,0)
    {\small Hutchinson};

\draw[
    ibmorange,
    line width=1.7pt,
    dashed
]
(3.1,0) -- (3.6,0);

\node[anchor=west] at (3.61,0)
    {\small Hutch\#};

\draw[
    ibmgold,
    line width=1.5pt
]
(5.55,0) -- (6,0);
\filldraw[
    ibmgold
]
(5.775,0) circle (1.5pt);
\node[anchor=west] at (6.01,0)
    {\small Hutch\# (Optimally Weighted)};

\draw[
    ibmmagenta,
    line width=1.7pt,
    dashed
]
(11.4,0) -- (11.9,0);
\node[anchor=west] at (11.91,0)
    {\small Hutch\# (Weighted)};

\end{tikzpicture}

\caption{
Relative RMSE for Frobenius norm estimation on matrices with singular values
$\sigma_j(A)\propto j^{-c}$. These plots confirm the improved $1/r^2$ scaling of the variance of Hutch\# compared to the $1/r$ scaling of Hutchinson's, and show that our weighted Hutch\# estimator from \Cref{sec:weighted} uniformly outperforms both estimators for any given number of matvecs.
}
\label{fig:weighted_decay}

\end{figure}

We test all three methods on four matrices with differing rates of spectral decay. In particular, each matrix is chosen to be a $1000 \times 1000$ matrix with $j$th singular value $\sigma_j = j^{-c}$. We test $c$ values in $\{0.5, 1, 1.5, 2\}$. All methods are implemented using Gaussian random vectors, so, by rotational invariance, have the same error distribution on any matrix with the same singular values. It thus suffices to use diagonal matrices in our experiments. 

Results are shown in \Cref{fig:weighted_decay}, which shows the normalized root mean squared error (RMSE) of each method over 200 trials, plotted against the number of matrix-vector products used. In the log-log plot, the improved $1/r^2$ scaling of the variance of Hutch\# is evident in comparison to the $1/r$ scaling of Hutchinson's. The method significantly outperforms Hutchinson's for matrices with spectral decay. However, as discussed in \Cref{sec:weighted}, Hutchinson's can perform better for matrices with flatter spectra. This is most evident in the plot for $c=0.5$. 

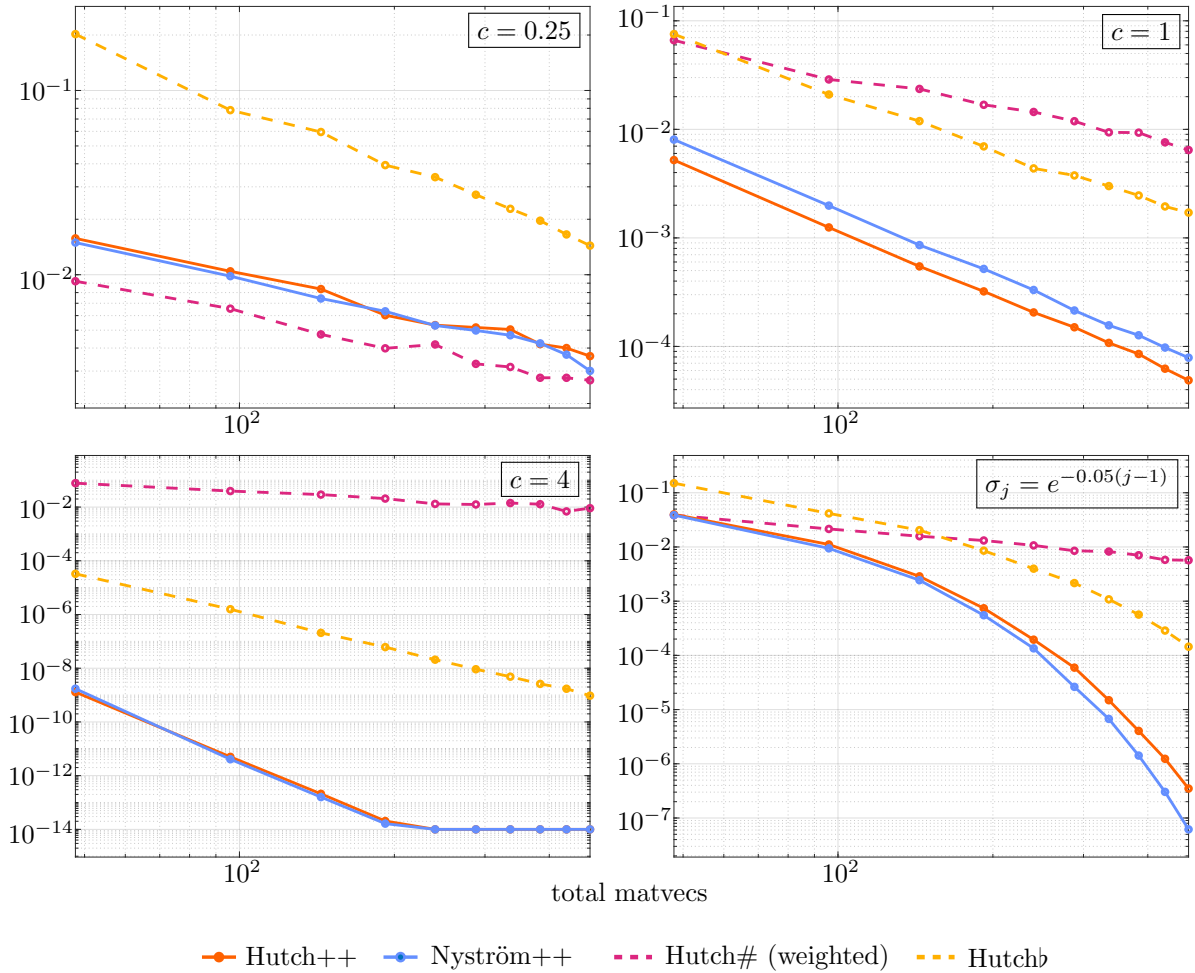
\begin{figure*}[b!]
\centering

\setlength{\tabcolsep}{14pt}
\renewcommand{\arraystretch}{1.0}

\begin{tabular}{cc}

\begin{overpic}[width=0.44\textwidth]
    {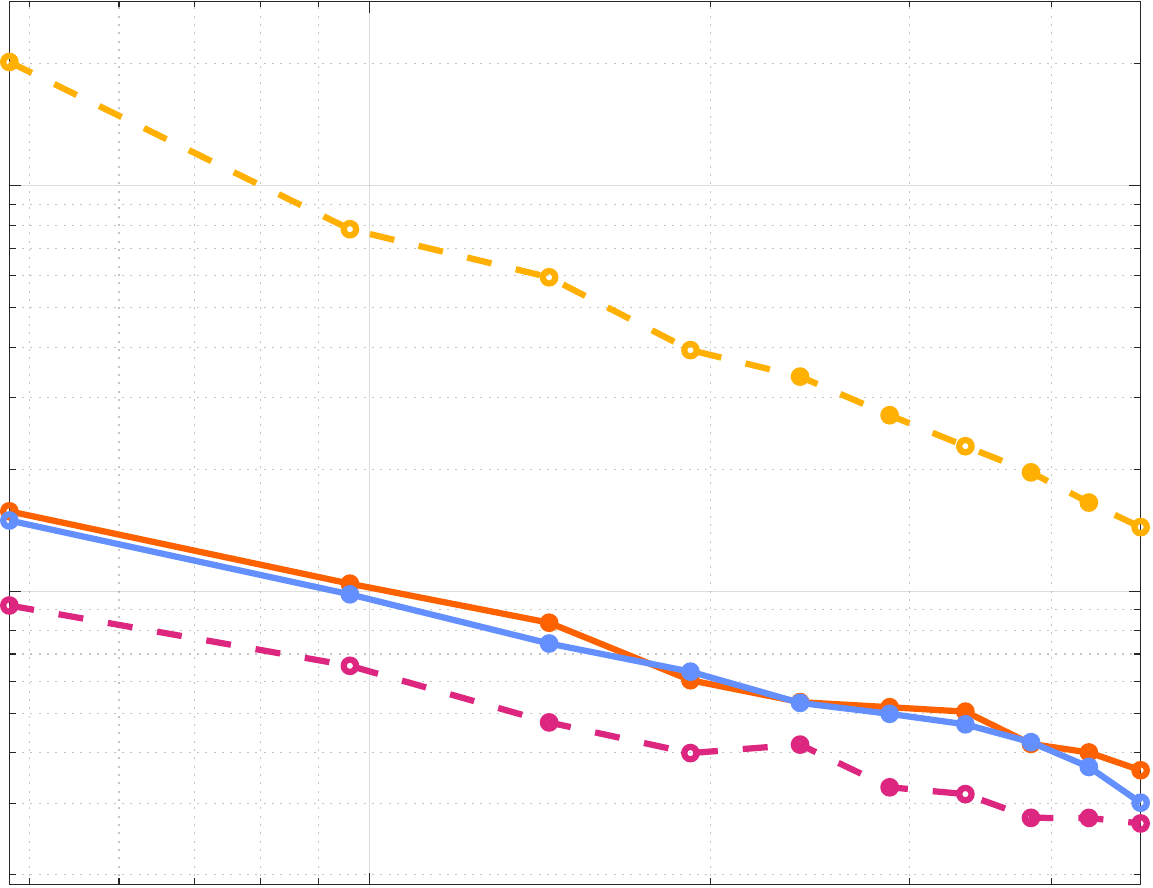}

    \put(75.75,70.75){\fbox{\small $c=0.25$}}

    \put(-10,59){\small $10^{-1}$}
    \put(-10,23){\small $10^{-2}$}

    \put(29,-4.5){\small $10^{2}$}

\end{overpic}
&
\begin{overpic}[width=0.44\textwidth]
    {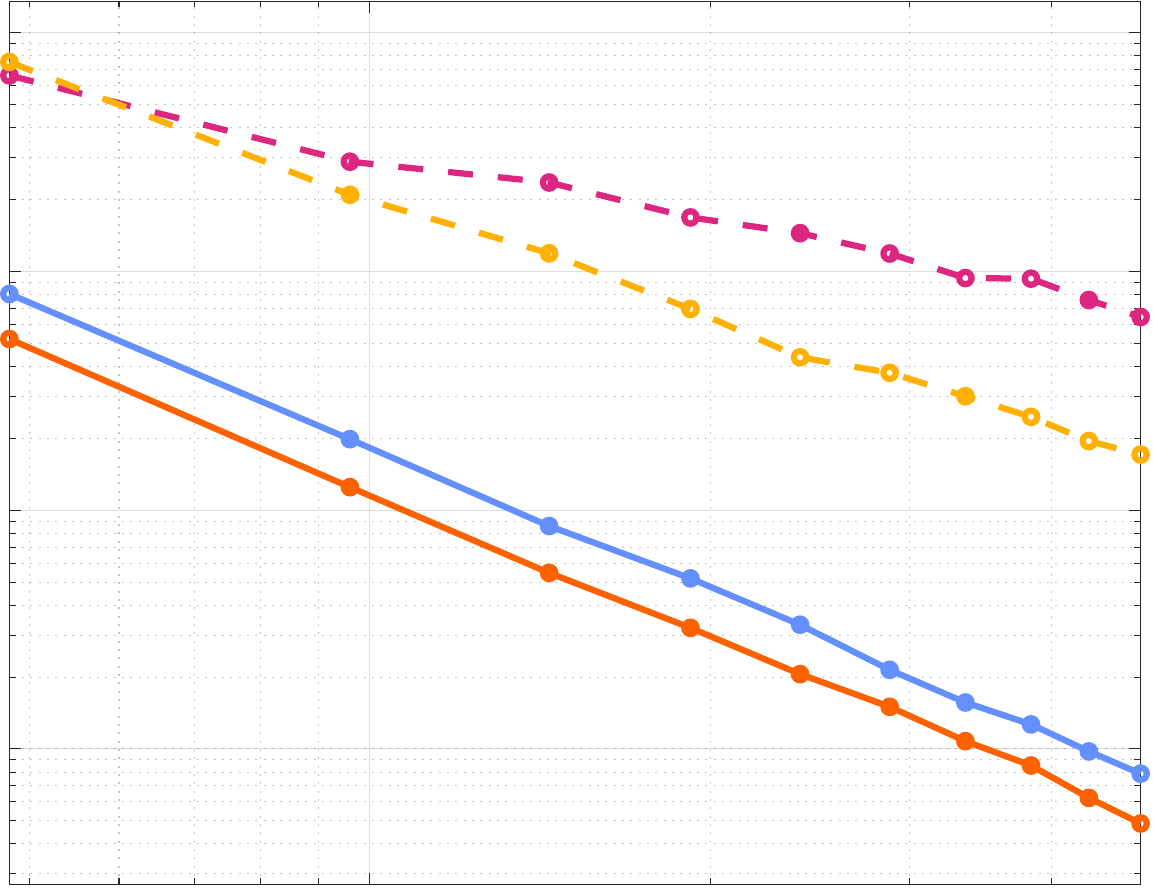}

    \put(82.5,70.5){\fbox{\small $c=1$}}

    \put(-10,73){\small $10^{-1}$}
    \put(-10,51){\small $10^{-2}$}
    \put(-10,31){\small $10^{-3}$}
    \put(-10,10){\small $10^{-4}$}

    \put(29,-4.5){\small $10^{2}$}

\end{overpic}

\\[1.2em]

\begin{overpic}[width=0.44\textwidth]
    {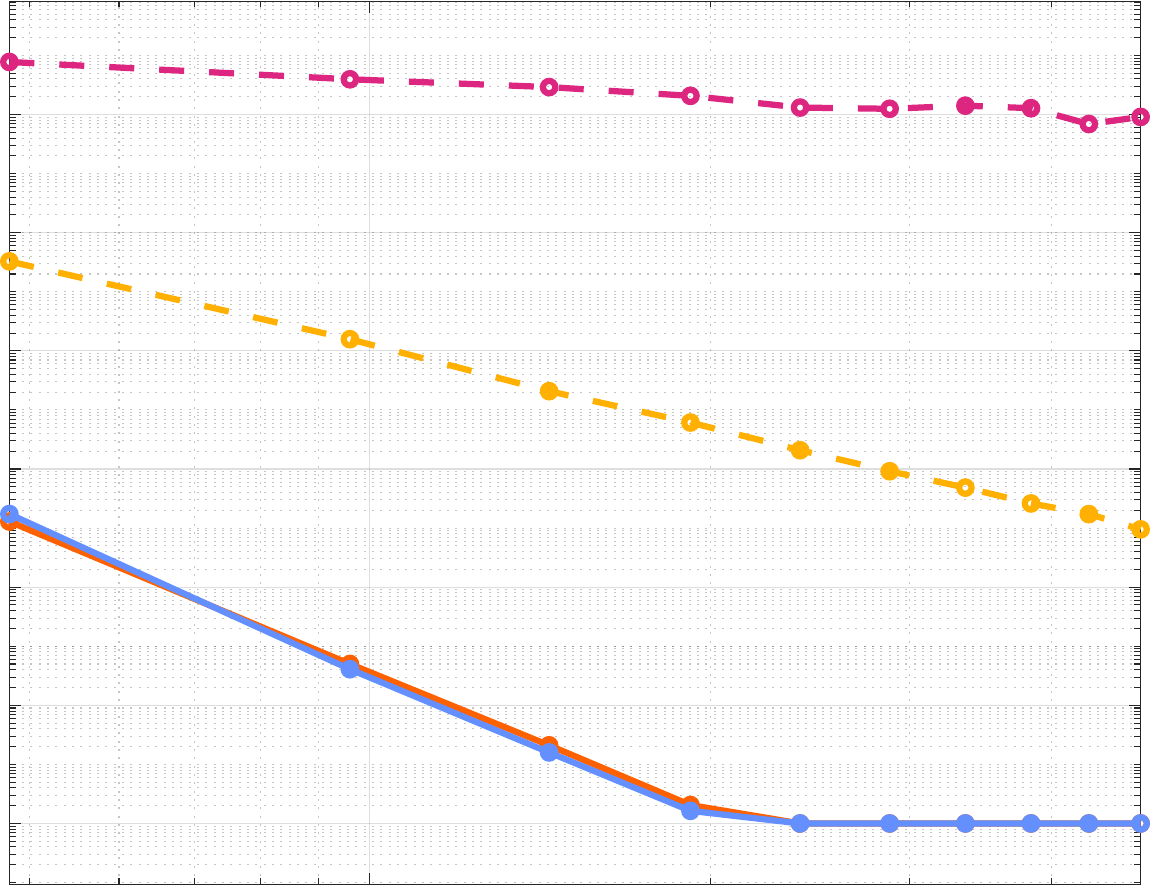}

    \put(82,70.75){\fbox{\small $c=4$}}

    \put(-10,65){\small $10^{-2}$}
    \put(-10,55){\small $10^{-4}$}
    \put(-10,45){\small $10^{-6}$}
    \put(-10,34){\small $10^{-8}$}
    \put(-12,24){\small $10^{-10}$}
    \put(-12,14){\small $10^{-12}$}
    \put(-12,3){\small $10^{-14}$}
    \put(29,-4.5){\small $10^{2}$}

\end{overpic}
&
\begin{overpic}[width=0.44\textwidth]
    {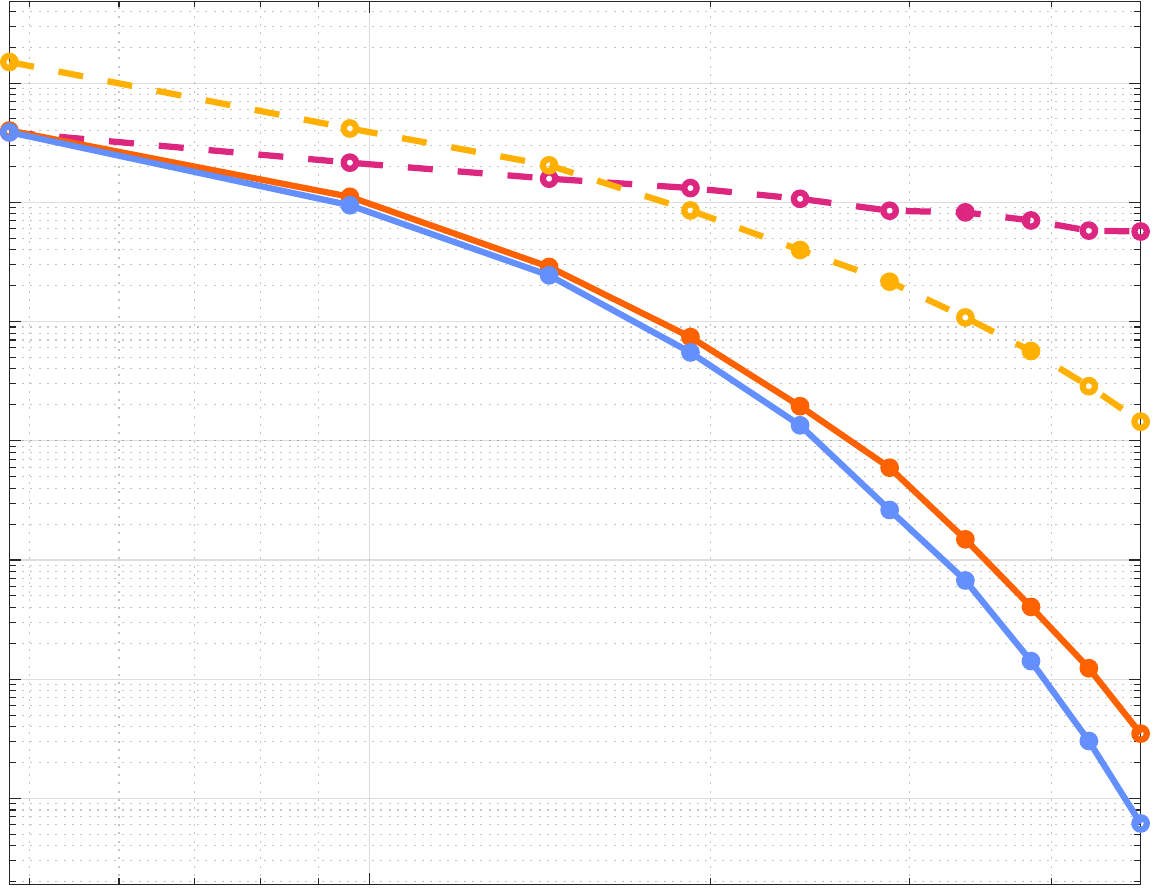}

    \put(58.5,70){\fbox{\small $\sigma_j=e^{-0.05(j-1)}$}}

    \put(-10,68){\small $10^{-1}$}
    \put(-10,57){\small $10^{-2}$}
    \put(-10,47){\small $10^{-3}$}
    \put(-10,37){\small $10^{-4}$}
    \put(-10,26){\small $10^{-5}$}
    \put(-10,16){\small $10^{-6}$}
    \put(-10,5.5){\small $10^{-7}$}

    \put(29,-4.5){\small $10^{2}$}

\end{overpic}

\end{tabular}


\vspace{0.35em}

{\small total matvecs}


\vspace{0.7em}

\begin{tikzpicture}[baseline]

\draw[
    ibmorange,
    line width=1.3pt
]
(0,0) -- (0.45,0);

\draw[
    ibmorange,
    line width=1pt,
    fill=ibmorange
]
(0.225,0) circle (1.5pt);

\node[anchor=west] at (0.44,0)
    {\small Hutch++};

\draw[
    ibmblue,
    line width=1.3pt
]
(2.4,0) -- (2.85,0);

\draw[
    ibmblue,
    line width=1pt,
    fill=matlabblue
]
(2.625,0) circle (1.5pt);

\node[anchor=west] at (2.89,0)
    {\small Nystr\"om++};

\draw[
    ibmmagenta,
    line width=1.7pt,
    dashed
]
(5.45,0) -- (5.95,0);

\node[anchor=west] at (5.99,0)
    {\small Hutch\# (weighted)};

\draw[
    ibmgold,
    line width=1.7pt,
    dashed
]
(9.45,0) -- (9.95,0);

\node[anchor=west] at (9.99,0)
    {\small Hutch$\flat$};

\end{tikzpicture}

\caption{
Relative RMSE for Frobenius norm estimation on matrices with polynomially decaying singular values
$\sigma_j(\bm A)\propto j^{-c}$ for $c\in\{0.25,1,4\}$, and for an
exponentially decaying spectrum $\sigma_j(\bm A)=e^{-0.05(j-1)}$.
These plots confirm that, as expected from our theoretical guarantees, Hutch$\flat$ can outperform Hutch\# on matrices with sufficient spectral decay. They also confirm the advantage of adaptivity: the adaptive Hutch++ and Nystr\"{o}m++ methods typically offer the lowest error for a given number of matrix-vector products.
}
\label{fig:hutchflat_mixed_spectra}

\end{figure*}

Fortunately, our suggested fix for this issue works perfectly. The weighted Hutch\# estimator from \Cref{sec:weighted} uniformly outperforms both Hutch\# and Hutchinson's for every matrix and every target number of matvecs we tried. We implemented the weighted estimator using the ``sketch reuse'' strategy discussed in \Cref{sec:weighted}, i.e., the same random Gaussian vectors are used to estimate the mixing parameter, $\alpha$, as are used to construct the Hutch\# estimate. We compared this strategy to an ``oracle'' version of the weighted estimator, where we explicitly compute the optimal mixing parameter $\alpha^*$. In all cases, the performance difference between the implemented estimator and the optimal estimator was negligible, suggesting that the sketch reuse strategy effectively finds a near-optimal mixing parameter.

\subsection{Hutch$\flat$ and Comparison to Hutch++}
Having established that the weighted Hutch\# estimator is the best ``simple'' non-adaptive Frobenius norm estimator, we compare the method to the more involved Hutch$\flat$ estimator from \Cref{sec:gennystrompp_variance}. We also compare both methods to state-of-the-art \emph{adaptive matvec} methods, including Hutch++ \cite{hutchpp} and Nystr\"om++ \cite{AHutchpp}, a variant of Hutch++ that uses the Nystr\"om method for low-rank approximation. Both methods achieve the same $O(1/\varepsilon)$ matvec complexity for $(1\pm \varepsilon)$ relative error Frobenius norm estimation as Hutch\#.

Importantly, both Hutch++ and Nystr\"om++ are designed for the more general problem of trace estimation. For Frobenius norm estimation, they are applied to the   matrix $\bm{A}^\T\bm{A}$. We note that doing so in a completely naive way can waste matvecs. In particular, a direct instantiation of the pseudocode in \cite{hutchpp} requires drawing two random matrices $\bm{\Omega}, \bm{G} \sim \gaussian(n,r)$ and then performing $3r$ matvecs with $\bm{A}^\T\bm{A}$, for a total of $6r$ matvecs. We instead implement Hutch++ using the mathematically equivalent formula $\|\bm A\bm Q\|_\F^2 + \frac{1}{r} \|\bm A (\bm{I} - \bm Q\bm Q^\T)\bm G\|_\F^2$, where $\bm Q=\operatorname{orth}(\bm A^\T\bm A\bm\Omega)$.
This more efficient implementation uses $3r$ matvecs with $\bm{A}$ and $r$ matvecs with $\bm{A}^\T$, for a total of $4r$. Similarly, we implement Nystr\"om++ using the formula $\|\bm Q^\T \bm{A}\|_\F^2 + \frac{1}{r} \|(\bm{I} - \bm Q\bm Q^\T)\bm A \bm G\|_\F^2$, where $\bm{Q} = \operatorname{orth}(\bm{A} \bm{\Omega})$ and $\bm{\Omega}, \bm{G} \sim \gaussian(n,r)$, which only requires $2r$ matvecs with $\bm{A}$ and $r$ matvecs with $\bm{A}^\T$. This formulation can be shown to be equivalent to the original Nyström++ formulation using, e.g., \cite[Lemma 1]{gittensmahoney}.

Again, we test the methods on $1000 \times 1000$ diagonal matrices with a range of spectral decay. The first three matrices satisfy $\sigma_j = j^{-c}$, for $c$ in $\{ 0.25, 1, 4\}$. The fourth matrix has exponential spectral decay: $\sigma_j = \exp(-0.05(j-1))$. We again report the normalized root mean squared error over 200 trials, with results shown in \Cref{fig:hutchflat_mixed_spectra}. For Hutch$\flat$, we use the same sketch allocation as suggested in \Cref{sec:gennystrompp_variance}: $2r$ columns for $\bm\Omega$, $4r$ for $\bm\Psi$, and $2r$ for $\bm G$.

As expected from \Cref{thm:gnpp_variance}, which shows that the variance of Hutch$\flat$ scales with $\|\bm{A} - [\bm{A}]_r\|_\F^2$, Hutch$\flat$ outperforms Hutch\# on matrices with sufficient spectral decay. The plot also shows the advantage of adaptivity in Frobenius norm estimation: while all of the methods tested enjoy the same $O(1/\varepsilon)$ worst-case matvec guarantee, the adaptive Hutch++ and Nystr\"{o}m++ methods perform markedly better in experiments. However, as discussed in \Cref{sec:intro}, such methods may not be applicable in some applications, or might have a higher cost per matvec than the non-adaptive Hutch\# and Hutch$\flat$ algorithms.


\subsection*{Acknowledgements}
DH and CM were supported by the NSF Algorithmic Foundations Program under awards 2427363 and 2045590.

\clearpage
\bibliographystyle{siam}
\bibliography{bibliography}

\appendix
\section{Proof of the Generalized Nystr\"om Error Bound}
\label{app:gn_gram}
In this section, we provide the proof of \cref{lemma:gn_gram}, which is restated below.

\GNGramLemma*

\begin{proof}
Let $\bm Q$ have orthonormal columns spanning $\range(\bm A\bm\Omega)$, and let $\bm Q_\perp$ have orthonormal columns spanning its orthogonal complement.
Write
\begin{align*}
    \bm P := \bm Q\bm Q^\T,\qquad
    \bm E := \bm B-\bm P\bm A.
\end{align*}
Expanding our error as $\bm A^\T\bm A-\bm B^\T\bm B
    =\bm A^\T(\bm{I}-\bm{P})\bm A
      -\bm A^\T\bm P\bm E
      -\bm E^\T\bm P\bm A
      -\bm E^\T\bm E$ and 
applying triangle inequality, we have:
\begin{align}\label{eq:gnpp_gram_split}
    \|\bm A^\T\bm A-\bm B^\T\bm B\|_\F^2
    &\leq \left(\|\bm A^\T(\bm{I}-\bm{P})\bm A\|_\F
    +2\|\bm A^\T\bm P\bm E\|_\F
    +\|\bm E^\T\bm E\|_\F\right)^2 \nonumber\\
    &\leq 3\|\bm A^\T(\bm{I}-\bm{P})\bm A\|_\F^2
    +12\|\bm A^\T\bm P\bm E\|_\F^2
    +3\|\bm E^\T\bm E\|_\F^2.
\end{align}
We bound the expectations of each term in \cref{eq:gnpp_gram_split} separately.

For the first term, note that $\|\bm A^\T(\bm I-\bm Q\bm Q^\T)\bm A\|_\F^2 = \|\bm A-\bm Q\bm Q^\T\bm A\|_{(4)}^4$, which is the Schatten-4 norm error of the standard Randomized SVD approximation to $\bm{A}$. A bound on this error can be found in \cite[Thm. 8.7]{tropp2023randomized}. Specifically, for $r\geq16$, the squared coefficient in that theorem is $(1+(r+1)/(r-3))^2\leq(30/13)^2\leq16/3$, so we obtain:
\begin{align}\label{eq:gnpp_range_schatten4}
\mathbb E\|\bm A^\T(\bm I-\bm Q\bm Q^\T)\bm A\|_\F^2
\leq \frac{16}{3}\left(
\|\bm A-[\bm A]_r\|_{(4)}^2
+\frac{1}{\sqrt r}\|\bm A-[\bm A]_r\|_\F^2
\right)^2.
\end{align}

For the remaining terms, we may assume $\operatorname{rank}(\bm A)\geq2r$; otherwise we have $\bm B=\bm A$ almost surely and the bound is immediate.
Thus $\bm Q$ has $2r$ columns almost surely.
We write:
\begin{align*}
\bm\Psi_1&=\bm\Psi^\T\bm Q\sim\gaussian(4r,2r), & &\text{and} &
\bm\Psi_2&=\bm\Psi^\T\bm Q_\perp\sim\gaussian(4r,m-2r),
\end{align*}
where $\bm Q_\perp$ has orthonormal columns spanning $\bm{Q}$'s complement. $\bm\Psi_1$ and $\bm\Psi_2$ are independent.

We will use the following fact, which is shown, e.g., in \cite[Proof of Lem.~5.1]{ChenDumanKelesHalikiasMuscoMuscoPersson:2025}:
\begin{equation}\label{eq:gnpp_error_identity}
    \bm Q^\T\bm E=\bm\Psi_1^\dagger\bm\Psi_2\bm Q_\perp^\T \bm{A}.
\end{equation}
We also use a Frobenius norm Randomized SVD expected error bound \cite[Proof of Thm.~10.5]{rsvd}:
\begin{align}\label{eq:gnpp_range_second}
    \mathbb E\|\bm Q_\perp^\T\bm A\|_\F^2
    =\mathbb E\|\bm A-\bm Q\bm Q^\T\bm A\|_\F^2
    &\leq\left(1+\frac{r}{r-1}\right)\|\bm A-[\bm A]_r\|_\F^2
    \leq\frac{31}{15}\|\bm A-[\bm A]_r\|_\F^2.
\end{align}
Using that $\bm P=\bm Q\bm Q^\T$ and \cref{eq:gnpp_error_identity}, we have
\begin{align*}
\bm A^\T\bm P\bm E
&=\bm A^\T\bm Q(\bm Q^\T\bm E)
=\bm A^\T\bm Q\bm\Psi_1^\dagger\bm\Psi_2\bm Q_\perp^\T\bm A.
\end{align*}
We condition on $\bm\Omega$ and $\bm\Psi_1$ and apply
\cite[eq.~(A.1a)]{perssonboullekressner2025}, which states that
$\mathbb E\|\bm C\bm\Gamma\bm D\|_\F^2=\|\bm C\|_\F^2\|\bm D\|_\F^2$ for fixed $\bm C,\bm D$ and a standard Gaussian $\bm\Gamma$. This gives
\begin{align*}
\mathbb E_{\bm\Psi_2}\!\left[\|\bm A^\T\bm P\bm E\|_\F^2
\middle|\bm\Omega,\bm\Psi_1\right]
&=\mathbb E_{\bm\Psi_2}\!\left[\|\bm A^\T\bm Q\bm\Psi_1^\dagger\bm\Psi_2\bm Q_\perp^\T\bm A\|_\F^2
\middle|\bm\Omega,\bm\Psi_1\right]\\
&=\|\bm A^\T\bm Q\bm\Psi_1^\dagger\|_\F^2
\|\bm Q_\perp^\T\bm A\|_\F^2.
\end{align*}
We can then apply \cite[eq.~(A.2a)]{perssonboullekressner2025}  to $\|\bm A^\T\bm Q\bm\Psi_1^\dagger\|_\F^2$ to obtain:
\begin{align*}
\mathbb E_{\bm\Psi_1}\!\left[\|\bm A^\T\bm Q\bm\Psi_1^\dagger\|_\F^2
\middle|\bm\Omega\right]
&=\frac{1}{2r-1}\|\bm Q^\T\bm A\|_\F^2.
\end{align*}
Averaging over $\bm\Omega$, we have:
\begin{align}\label{eq:gnpp_cross_term}
    \mathbb E\|\bm A^\T\bm P\bm E\|_\F^2 
    =\frac{1}{2r-1}\mathbb E_{\bm \Omega}\left[
        \|\bm Q^\T\bm A\|_\F^2\|\bm Q_\perp^\T \bm{A}\|_\F^2\right] 
     &\leq \frac{1}{2r-1}\|\bm A\|_\F^2\, \mathbb E_{\bm \Omega}\left[
        \|\bm Q_\perp^\T \bm{A}\|_\F^2\right] \nonumber\\
    &\leq\frac{1}{2r-1}\left(1+\frac{r}{r-1}\right)\|\bm A\|_\F^2\|\bm A-[\bm A]_r\|_\F^2\nonumber \\
     &\leq\frac{16}{15r} \|\bm A\|_\F^2\|\bm A-[\bm A]_r\|_\F^2.
\end{align}
The second to last step follows from \cref{eq:gnpp_range_second} and the last from using that $r \geq 16$.

It remains to bound $\mathbb E\|\bm E^\T\bm E\|_\F^2$. Observe that $\bm{E} = \bm{B} - \bm{P}\bm{A}$ has columns in the span of $\bm{Q}$. 
Hence $\bm E=\bm Q\bm Q^\T\bm E$ and $\bm E^\T\bm E=(\bm Q^\T\bm E)^\T(\bm Q^\T\bm E)$.
We thus have:
\begin{equation}\label{eq:gnpp_EtE_schatten}
    \|\bm E^\T\bm E\|_\F^2
    =\|\bm Q^\T\bm E\|_{(4)}^4
    =\left\|\bm\Psi_1^\dagger\bm\Psi_2\bm Q_\perp^\T\bm A\right\|_{(4)}^4.
\end{equation}
We next apply a Schatten-4 analogue of an identity used early: for fixed matrices $\bm C$ and $\bm D$ and a standard Gaussian matrix $\bm\Gamma$ of compatible size, \cite[eq.~(A.1b)]{perssonboullekressner2025} gives
\begin{equation}\label{eq:pbk_schatten4}
    \mathbb E\left[\|\bm C\bm\Gamma\bm D\|_{(4)}^4\right]
    =\|\bm C\|_{(4)}^4\|\bm D\|_{(4)}^4
    +\|\bm C\|_\F^4\|\bm D\|_{(4)}^4
    +\|\bm C\|_{(4)}^4\|\bm D\|_\F^4.
\end{equation}
Conditioning on $\bm\Omega$ and $\bm\Psi_1$, we can apply \cref{eq:pbk_schatten4} to \cref{eq:gnpp_EtE_schatten} to obtain:
\begin{align}\label{eq:gnpp_EtE_psi2}
    \mathbb E_{\bm\Psi_2}\left[\|\bm E^\T\bm E\|_\F^2\middle|\bm\Omega,\bm\Psi_1\right]
    ={}&\left(\|\bm\Psi_1^\dagger\|_{(4)}^4+\|\bm\Psi_1^\dagger\|_\F^4\right)
    \left\|\bm Q_\perp^\T\bm A\right\|_{(4)}^4
    +\|\bm\Psi_1^\dagger\|_{(4)}^4\left\|\bm Q_\perp^\T\bm A\right\|_\F^4.
\end{align}
Exact values for $\|\bm\Psi_1^\dagger\|_{(4)}^4$ and $\|\bm\Psi_1^\dagger\|_\F^4$ are provided in \cite[eqs.~(A.4a) and (A.4b)]{perssonboullekressner2025}. We apply those identities, noting that $\bm\Psi_1$ has dimensions $4r \times 2r$. We obtain:
\begin{align*}
    \mathbb E\|\bm\Psi_1^\dagger\|_{(4)}^4
    &
    =\frac{4r-1}{(2r-1)(2r-3)} &&\text{and}&
    \mathbb E\|\bm\Psi_1^\dagger\|_\F^4
    &=
    \frac{4r^2-4r+2}{(2r-1)(2r-3)}.
\end{align*}
Substituting in into \cref{eq:gnpp_EtE_psi2} and using that $r \geq 16$, we obtain:
\begin{align}\label{eq:gnpp_EtE_omega}
    \mathbb E\!\left[\|\bm E^\T\bm E\|_\F^2\middle|\bm\Omega\right]
    &=\frac{4r^2+1}{(2r-1)(2r-3)}\|\bm Q_\perp^\T\bm A\|_{(4)}^4
    +\frac{4r-1}{(2r-1)(2r-3)}\|\bm Q_\perp^\T\bm A\|_\F^4\nonumber\\
    &\leq\frac87\|\bm Q_\perp^\T\bm A\|_{(4)}^4+\frac{9}{8r}\|\bm Q_\perp^\T\bm A\|_\F^4.
\end{align}
Finally, we average \cref{eq:gnpp_EtE_omega} over $\bm\Omega$, handling the two terms differently. This first term equals:
\begin{align*}
    \|\bm Q_\perp^\T\bm A\|_{(4)}^4
    =\|\bm A^\T\bm Q_\perp\bm Q_\perp^\T\bm A\|_\F^2
    =\|\bm A^\T(\bm I-\bm P)\bm A\|_\F^2,
\end{align*}
which is precisely the quantity bounded by \cref{eq:gnpp_range_schatten4}.
For the second term, we use $\|\bm Q_\perp^\T\bm A\|_\F^2\leq\|\bm A\|_\F^2$ to write $\|\bm Q_\perp^\T\bm A\|_\F^4\leq\|\bm A\|_\F^2\|\bm Q_\perp^\T\bm A\|_\F^2$ and then apply \cref{eq:gnpp_range_second}.
Plugging into \cref{eq:gnpp_EtE_omega}, we obtain:
\begin{align}\label{eq:gnpp_quadratic_term}
    \mathbb E\|\bm E^\T\bm E\|_\F^2
    \leq \frac{128}{21}\left(\|\bm A-[\bm A]_r\|_{(4)}^2
    +\frac{\|\bm A-[\bm A]_r\|_\F^2}{\sqrt r}\right)^2
    +\frac{93}{40r}\|\bm A\|_\F^2\|\bm A-[\bm A]_r\|_\F^2.
\end{align}
Combining \cref{eq:gnpp_gram_split}, \cref{eq:gnpp_range_schatten4}, \cref{eq:gnpp_cross_term}, and \cref{eq:gnpp_quadratic_term}, and rounding the coefficients $3\cdot(16/3)+3\cdot(128/21)=240/7\leq35$ and $12\cdot(16/15)+3\cdot(93/40)=791/40\leq20$ proves \cref{eq:gnpp_gram_bound}.
\end{proof}

\end{document}